\documentclass{amsart}

\newtheorem{theorem}{Theorem}[section]
\newtheorem{lemma}[theorem]{Lemma}

\theoremstyle{definition}
\newtheorem{definition}[theorem]{Definition}

\newtheorem{corollary}[theorem]{Corollary}

\theoremstyle{remark}

\numberwithin{equation}{section}

\begin{document}

\title[ $\beta$-almost solitons on almost co-k\"{a}hler manifolds]{$\beta$-almost solitons on almost co-k\"{a}hler manifolds}

\author{Shahroud Azami}
\address{Department of Mathematics, Faculty of Science
Imam Khomeini International University,
Qazvin, Iran. }

\email{azami@sci.ikiu.ac.ir}



\subjclass[2010]{53C15, 53C25,  53C44}



\keywords{Ricci soliton, Yamabe soliton, Contact manifold, Almost co-K\"{a}hler manifold.}
\begin{abstract}
The object of the present paper is to study $\beta$-almost Yamabe solitons and  $\beta$-almost Ricci solitons on almost co-K\"{a}hler manifolds. In this paper, we prove that  if an almost co-K\"{a}hler manifold  $M$ with the Reeb vector field $\xi$ admits a $\beta$-almost Yamabe solitons  with the potential vector field  $\xi$  or $b\xi$, where $b$ is  a smooth function then manifold is $K$-almost co-K\"{a}hler manifold or the soliton is trivial, respectively. Also, we show if a  closed $(\kappa,\mu)$-almost co-K\"{a}hler manifold  with $n>1$ and $\kappa<0$ admits a $\beta$-almost Yamabe soliton    then the soliton is trivial and expanding. Then  we study
an   almost co-K\"{a}hler manifold    admits a $\beta$-almost Yamabe soliton or  $\beta$-almost Ricci  soliton  with $V$ as the potential vector field, $V$ is a  special geometric   vector field.
\end{abstract}

\maketitle
\section{Introduction}
Over the last few years,  the geometric  flows have been an interesting topic of active research in both  mathematics and physics.  The Ricci flow was introduced by Hamilton  \cite{H}, which it is an evolution equation for metrics on a Riemannian manifold defined as follows
\begin{equation}
\frac{\partial g}{\partial t}=-2S, \,\,\,\,\,\,\,\,\ g(0)=g_{0},
\end{equation}
where $S$ denotes the Ricci tensor. A Ricci soliton $(M,g, V, \lambda)$ on a Riemannian manifold $(M,g)$ is a special solution  to the Ricci flow,  a generalization of  an Einstein metric and is defined by
\begin{equation}\label{rs}
\mathcal{L}_{V}g+2S+2\lambda g=0,
\end{equation}
where $\mathcal{L}_{V}$ is the Lie derivative operator  along the vector field $V$ (called the potential vector field )  on  $M$ and $\lambda$ is a real number. The Ricci soliton is said to be shrinking, steady, or expanding  according as $\lambda$ is negative, zero,  and positive, respectively.  If the vector field $V$ is the gradient of a potential function $f$, then  $g$ is called a gradient Ricci soliton.
If in Ricci soliton (\ref{rs}), $\lambda$ be a function $\lambda:M\to \mathbb{R}$ then it called almost  Ricci soliton. In \cite{SP, GHA} have been investigated the almost Ricci soliton. Also, if there exist a function  $\beta:M\to \mathbb{R}$ such that
\begin{equation}\label{ars}
\beta\mathcal{L}_{V}g+2S+2\lambda g=0 ,
\end{equation}
 then this soliton is  called $\beta$-almost Ricci soliton which studied in \cite{GHA, JNG}. The $\beta$-almost Ricci soliton  is called expanding, steady  and shrinking when $\lambda >0$, $\lambda=0$ or $\lambda<0$, respectively. A $\beta$-almost Ricci soliton is said to be trivial if the potential vector field $V$ is  homotherhetic, i.e., $\mathcal{L}_{V}g=cg$, for some constant $c$. Otherwise, it is called non-trivial.
 Similarly, a Riemannian manifold $(M,g)$  is said to be  $\beta$-almost Yamabe soliton  if there exist a vector field $V$ on $M$  (called the potential vector field ),  a soliton function $\lambda: M\to \mathbb{R}$  and a smooth function  $\beta:M\to \mathbb{R}$ such that
\begin{equation}\label{ys}
\beta\mathcal{L}_{V}g=(\lambda-r)g,
\end{equation}
where $r$ is scalar curvature of $M$ respect to metric $g$. When the function $\lambda$  in (\ref{ars}) (  and in (\ref{ys}) ) is constant  we simply say that it is  a  $\beta$-Ricci soliton (   a  $\beta$-Yamabe soliton). The $\beta$-almost Yamabe soliton  is called expanding, steady or shrinking when $\lambda <0$, $\lambda=0$ or $\lambda>0$, respectively. A $\beta$-almost Yamabe soliton is said to be trivial if the potential vector field $V$ is  Killing, i.e., $\mathcal{L}_{V}g=0$.
We say that  $\beta$ is defined signal whenever either $\beta>0$ on $M$ or $\beta<0$ on $M$.
 During the last two decades, the geometry of Ricci soliton and other solitons have been  the focus of  attention of many mathematicians and  physicists.  For  theoretical physicists the  Ricci solitons are as quasi Einstein metrics and they have been  looking into the equation of geometric sloitons in relation with topics of physics as String theory and general relativity \cite{F, W}.\\

In contact geometry,  gradient Ricci soliton have been studied by Sharma \cite{S} as a $K$-contact and by Ghosh et al. \cite{G}
as a  $(\kappa, \mu)$-metric. In \cite{C} Cho et al. and in \cite{KH} Hui et al. studied $\eta$-Ricci solitons on  real hypersurfaces in a non-flat complex space form and on  $\eta$-Einstein Kenmotsu manifolds, respectively. Also, in \cite{JS} Suh et al. investigated  the Yamabe solitons and Ricci solitons on almost co-K\"{a}hler manifolds. In \cite{DKP}, Kar and Majhi studied $\beta$-almost Ricci soliton on almost co-K\"{a}hler manifolds with $\xi$ belong  to $(\kappa, \mu)$-nullity distribution.\\

Motivated by the above studies the object of present paper is to study $\beta$-almost Yamabe solitons and  $\beta$-almost Ricci solitons on almost co-K\"{a}hler manifolds and we  generalize the results of  \cite{JS}  and also we obtain some other results of these  solitons  on almost co-K\"{a}hler manifolds when the potential vector field of solitons satisfies in certain conditions. This paper is organized as follows. In section \ref{s2}, after a brief introduction, we study almost co-K\"{a}hler manifolds and give some formula that will be used in the next sections. In section \ref{s3}, we consider $\beta$-almost Yamabe solitons on almost co-K\"{a}hler manifolds and prove if an almost co-K\"{a}hler manifold  $M$ with the Reeb vector field $\xi$ admits a $\beta$-almost Yamabe solitons  with the potential vector field  $\xi$  or $b\xi$, where $b$ is  a smooth function, and $\beta$ is defined signal, then manifold $M$ is $K$-almost co-K\"{a}hler manifold or the soliton is trivial, respectively. Also we prove  several important results about the geometric fields and $\beta$-almost Yamabe solitons on  almost co-K\"{a}hler manifolds.  In follow of this section, we study  $\beta$-almost Yamabe solitons on $(\kappa, \mu)$-almost co-K\"{a}hler manifolds. In the last section, we consider $\beta$-almost Ricci solitons  with geometric vector fields on  almost co-K\"{a}hler manifolds.
\section{Preliminaries}\label{s2}
In this section, we give some well known definitions and formulae on almost co-K\"{a}hler manifolds which will be useful in the later sections.
A smooth $(2n+1)$ dimensional Riemannian manifold $(M,g)$ is said to admit an almost  contact metric structures $(\phi,\xi, \eta, g)$, if it admits a $(1,1)$ tensor field $\phi$, a unit vector field $\xi$ (called the Reeb vector field), and  a $1$-form $\eta$ satisfying \cite{EB, DEB},
\begin{equation}\label{fc}
\phi^{2}=-I+\eta\otimes \xi,\,\,\,\,\,\,\eta(\xi)=1,\,\,\,\,\,\,\,\phi\xi=0,\,\,\,\,\eta\circ\phi=0,
\end{equation}
and
\begin{equation}\label{fc1}
g(\phi X, \phi Y)=g(X,Y)-\eta(X)\eta(Y),
\end{equation}
or equivalently
\begin{equation}\label{fc2}
g(\phi X, Y)=-g(X,\phi Y)\,\,\,\,\,\,\,\text{and}\,\,\,\,\,\,g(X,\xi)=\eta(X),
\end{equation}
for all vector fields $X,Y$ on $M$. For an almost contact metric manifold  $(M^{2n+1},\phi,\xi, \eta, g)$, we can always define a $2$-form $\Phi$ as $\Phi(X,Y)=g(x,\phi Y)$. An almost contact metric structure becomes a contact metric structure if $\Phi=d\eta$. In this case, $1$-form $\eta$ is a contact form, $\xi$ is its characteristic vector field, and  $\Phi$ is the fundamental $2$-form. If, in addition $\xi$ is a Killing vector field, then $M^{2n+1}$ is  said to be $K$-contact manifold.  The almost contact metric structure is said to be normal if $[\phi,\phi]=-2d\eta\otimes\xi$  where
\begin{equation*}
[\phi,\phi](X,Y)=[\phi X,\phi Y]+\phi^{2}[X,Y]-\phi[\phi X,Y]-\phi[X,\phi Y],
\end{equation*}
for any vector fields $X,Y$ on $M$.  A normal almost contact metric manifold is said to be Sasakian, that is an almost contact metric manifold is Sasakian if and only if
\begin{equation*}
(\nabla_{X}\phi)Y=g(X,Y)\xi-\eta(Y)X,
\end{equation*}
or equivalently
\begin{equation*}
R(x,y)\xi=\eta(Y)X-\eta(X)Y,
\end{equation*}
for any vector fields $X,Y$ on $M$ (see \cite{DEB}). An almost contact metric manifold $(M^{2n+1},\phi,\xi, \eta, g)$ is said to be almost co-K\"{a}hler manifold \cite{DEB, BC} if  both the $1$-form $\eta$ and  the $2$-form $\Phi$ are closed. If, in addition the associated almost contact structure is normal, which is  also equivalent to $\nabla \Phi=0$, or equivalently  $\nabla \phi=0$, then $M$ is said to be co-K\"{a}hler manifold. There exists some examples  of  (almost) co-K\"{a}hler manifolds, for instance, the Riemannian product  of a real line  and a (almost) K\"{a}hler manifold admits a (almost) co-K\"{a}hler structure (\cite{DC, CM, ZO, ZOL}). On an almost co-K\"{a}hler manifold $(M^{2n+1},\phi,\xi,\eta, g)$ we set $h=\frac{1}{2}\mathcal{L}_{\xi}\phi$ and $h'=h\circ \phi$. Then the following formulas also hold for a almost co-K\"{a}hler manifold \cite{PD, HE, ZO, ZOL}
\begin{eqnarray}\label{ck1}
&&h\xi=0,\,\,\,\,\,\,\,\,\,h\phi+\phi h=0,\,\,\,\,\,\,\,tr h=tr h',\\\label{ck2}
&&\nabla_{\xi}\phi=0,\,\,\,\,\nabla\xi=h',\,\,\,\,\,\,\,\,\,\,\, \,\,\,\,\,\,\,div \xi=0,\\\label{ck3}&&S(\xi, \xi)+||h||^{2}=0.
\end{eqnarray}
If, in addition, we put $l=R(.,\xi)\xi$, then we also get $\phi l\phi -l=2h^{2}$, where $R$ is the Riemannian curvature tensor. From the second term of (\ref{ck1}) it is easy to see that
\begin{equation}\label{ck4}
(\mathcal{L}_{\xi}g)(X,Y)=2g(h'X,Y),
\end{equation}
for any vector fields $X,Y$ on $M$. Therefore, the Reeb vector field $\xi$ on almost co-K\"{a}hler manifold  is Killing if and only if the $(1,1)$ tensor field $h$  vanishes. An almost co-K\"{a}hler manifold is said to be a $K$-almost co-K\"{a}hler manifold if the Reeb vector field $\xi$   is Killing.\\
A contact metric manifold  $(M^{2n+1},\phi,\xi,\eta, g)$ whose curvature tensor satisfies
\begin{equation}\label{n1}
R(X,Y)\xi=\kappa\left[ \eta(Y)X-\eta(X)Y\right]+\mu\left[  \eta(Y)hX-\eta(X)hY\right],
\end{equation}
for any vector fields $X,Y$ on $M$  and $\kappa, \mu\in\mathbb{R}$ is called $(\kappa,\mu)$-contact manifold and $\xi$ is said to belong to the $(\kappa,\mu)$-nullity distribution. Similarly, we have
\begin{definition}
An almost co-K\"{a}hler manifold $(M^{2n+1},\phi,\xi,\eta, g)$  is said to be  a  $(\kappa,\mu)$-almost co-K\"{a}hler manifold if the Reeb vector field $\xi$  satisfies the equation (\ref{n1}).
\end{definition}
In a consequence of  (\ref{n1}), we  obtain
\begin{equation}\label{q}
S(X,\xi)=2n\kappa \eta(X),
\end{equation}
and $Q\xi=2n\kappa\xi$, where  $Q$ is the Ricci operator defined by $g(QX,Y)=S(X,Y)$.
\section{$\beta$-almost Yamabe solitons on almost co-K\"{a}hler manifolds}\label{s3}
In this section we study   $\beta$-almost Yamabe solitons on almost co-K\"{a}hler manifolds $(M^{2n+1},\phi,\xi,\eta, g)$.  If  we assume that the potential vector field of soliton be $\xi$, then we have:
\begin{theorem}\label{t1}
If an almost co-K\"{a}hler manifold $(M^{2n+1},\phi,\xi,\eta, g)$ admits a $\beta$-almost Yamabe soliton  with $\xi$ as the potential vector field and $\beta$ is defined signal, then manifold $M$ is $K$-almost co-K\"{a}hler manifold.
\end{theorem}
\begin{proof}
Let  in almost co-K\"{a}hler manifold $(M^{2n+1},\phi,\xi,\eta, g)$ the metric $g$  be  a $\beta$-almost Yamabe soliton  with  the potential vector field $\xi$. Then, we have $\mathcal{L}_{\xi}g=\frac{\lambda -r}{\beta}g$. By take trace of both-sides of last identity we get $2 div \xi=\frac{\lambda -r}{\beta}(2n+1)$. Since $div \xi=0$, we conclude $\frac{\lambda -r}{\beta}=0$, this means  $\xi$ is  a Killing vector field. This completes the proof of Theorem.
\end{proof}
Any three dimensional almost co-K\"{a}hler manifold is co-K\"{a}hler manifold  if and only if  it is $K$-almost co-K\"{a}hler manifold  \cite{ZO}. Hence, we get the following result.
\begin{corollary}\label{c1}
If an almost co-K\"{a}hler manifold $(M^{3},\phi,\xi,\eta, g)$ admits a $\beta$-almost Yamabe soliton  with $\xi$ as the potential vector field and $\beta$ is defined signal, then manifold $M$ is a  co-K\"{a}hler manifold.
\end{corollary}
Now, if we assume that the potential vector field of  $\beta$-almost Yamabe soliton   is  pointwise  collinear  with the Reeb vector field, then we have the following theorem.
\begin{theorem}\label{t2}
If an almost co-K\"{a}hler manifold $(M^{2n+1},\phi,\xi,\eta, g)$ admits a $\beta$-almost Yamabe soliton  with $b\xi$ as the potential vector field where $b$ is non-zero smooth function and $\beta$ is defined signal, then the soliton is trivial.
\end{theorem}
\begin{proof}
Using (\ref{ck2}) we can write
\begin{equation*}
\nabla_{X}(b\xi)=X(b)\xi+b\nabla{X}\xi=X(b)\xi+bh'X,
\end{equation*}
for any vector field $X$ on $M$. On the other hand, the metric $g$  is  a $\beta$-almost Yamabe soliton with  the potential vector field $b\xi$, then we  have $\beta \mathcal{L}_{b\xi}g=(\lambda -r)g$ and this implies
\begin{eqnarray}\nonumber
(\lambda -r)g(X,Y)&=&\beta g(\nabla_{X}(b\xi),Y)+\beta g(X,\nabla_{Y}(b\xi))\\\label{1}
&=&\beta X(b)\eta(Y)+\beta Y(b)\eta(X)2\beta b g(h'X,Y).
\end{eqnarray}
For each point $p$ in $M$, we consider a local $\phi$-basis  $\{e_{i}:\,\, 1\leq i\leq 2n+1\}$ on the tangent space $T_{p}M$. Since
\begin{equation*}
\sum_{i=1}^{2n+1}g(h'e_{i},e_{i})=\sum_{i=1}^{2n+1}g(\nabla_{e_{i}}\xi,e_{i})=div \xi=0,
\end{equation*}
taking  $X=Y=e_{i}$ in (\ref{1}) and summing over $i$, we derive
\begin{equation}\label{2}
(\lambda -r)(2n+1)=2\beta \xi(b).
\end{equation}
Again, substituting $X=Y=\xi$ in (\ref{1}), we get
\begin{equation}\label{3}
(\lambda -r)=2\beta \xi(b).
\end{equation}
Equations (\ref{2})  and (\ref{3}) yield $\lambda=r$. Putting  $\lambda=r$ in $\beta \mathcal{L}_{b\xi}g=(\lambda -r)g$ , we obtain $ \mathcal{L}_{b\xi}g=0$ . Thus, $b\xi$ is a Killing vector field and the soliton is trivial.
\end{proof}
Now we state the following lemma, which will be used in the next results  when in a a $(\kappa,\mu)$-almost co-K\"{a}hler manifold  with $n>1$ constant $\kappa$ is negative.
\begin{lemma}(\cite{YW})\label{l1}
Let $(M^{2n+1},\phi,\xi,\eta, g)$ be a $(\kappa,\mu)$-almost co-K\"{a}hler manifold  with $n>1$ and $\kappa<0$. Then  the Ricci operator is given by
 \begin{equation}\label{y1}
Q=\mu h+2n\kappa \eta\otimes \xi,
\end{equation}
where $\kappa$ is a  constant and $\mu$ is a smooth function satisfying $d\mu\wedge \eta=0$.
\end{lemma}
\begin{theorem}\label{t3}
If a  closed $(\kappa,\mu)$-almost co-K\"{a}hler manifold $(M^{2n+1},\phi,\xi,\eta, g)$  with $n>1$ and $\kappa<0$ admits a $\beta$-almost Yamabe soliton   and $\beta$ is defined signal, then the soliton is trivial and expanding.
\end{theorem}
\begin{proof}
Let the metric $g$ of $(\kappa,\mu)$-almost co-K\"{a}hler manifold \linebreak $(M^{2n+1},\phi,\xi,\eta, g)$ be   a $\beta$-almost Yamabe soliton with  the potential vector field $V$, then we  have $\beta \mathcal{L}_{V}g=(\lambda -r)g$. Set $\rho:=\frac{\lambda -r}{\beta}$, from \cite{KY} we have
\begin{equation}\label{4}
\mathcal{L}_{V}r=-\rho r-2n \Delta \rho,
\end{equation}
where $\Delta=div \,grad$  denotes the Laplace operator of $g$. On the other hand, from (\ref{y1}) of lemma \ref{l1} we derive
\begin{equation}\label{5}
S(X,Y)=\mu g(hX,Y)+2n\kappa\eta(X)\eta(Y),
\end{equation}
for any vector fields $X,Y$ on $M$. Consider a local $\phi$-basis  $\{e_{i}:\,\, 1\leq i\leq 2n+1\}$ on the tangent space $T_{p}M$. Putting $X=Y=e_{i}$ in (\ref{5}) and summing over $i$, $1\leq i\leq 2n+1$, we conclude $r=2n\kappa$. Hence the scalar curvature $r$ is constant and negative. Thus $\mathcal{L}_{V}r=0$ and equation (\ref{4}) implies that $-\Delta \rho=\frac{r}{2n}\rho$. Multiplying both sides of this equation in function $\rho$,  integrating over $M$ and using divergence theorem we obtain
\begin{equation}\label{6}
\int_{M}|\nabla \rho|^{2}\Omega=\frac{r}{2n}\int_{M} \rho^{2}\Omega,
\end{equation}
where $\Omega$ is the volume form of $M$.
If $\rho\neq 0$ the(\ref{6}) implies  $r$ is positive and this is a contradiction. Therefore, $\rho=0$, i.e. $\lambda=r$. So, the soliton is trivial and expanding.
\end{proof}
Now we recall the following definition which be used in the next theorem.
\begin{definition}
A vector field $V$ on a contact manifold $(M^{2n+1},\phi,\xi,\eta, g)$ is called a contact vector field if it satisfies
\begin{equation}\label{cvf}
\mathcal{L}_{V}\eta=\psi \eta,
\end{equation}
for some smooth function $\psi$ on $M$. If $\psi=0$ on $M$, then the vector field $V$ is called  a strict contact vector field.
\end{definition}
\begin{theorem}\label{t7}
If an   almost co-K\"{a}hler manifold $(M^{2n+1},\phi,\xi,\eta, g)$   admits a $\beta$-almost Yamabe soliton    with $V$ as the potential vector field, $V$ is a  contact vector field, i.e. $\mathcal{L}_{V}\eta=\psi \eta$, then  $\psi$ is constant and this soliton is shrinking, steady or expanding according as $r+\frac{n+1}{2n+1}\beta \psi$ be negative, zero or positive, respectively.  Moreover, if $M$ be is closed manifold then $\lambda=r$.
\end{theorem}
\begin{proof}
Since the metric $g$ is $\beta$-almost Yamabe soliton    with  the potential vector field $V$, then
\begin{equation}\label{b1}
\beta (\mathcal{L}_{V} g)(X, Y)=(\lambda-r)g(X,Y).
\end{equation}
This implies
\begin{equation}\label{b2}
\beta\mathcal{L}_{V}(g(X,Y))-\beta g(\mathcal{L}_{V}X,Y)-\beta g(X, \mathcal{L}_{V}Y)=(\lambda-r)g(X,Y).
\end{equation}
Putting $\xi$ for $X$ and $Y$ in (\ref{b2}), we conclude
\begin{equation}\label{b3}
2\beta g(\mathcal{L}_{V}\xi,\xi)=\lambda-r.
\end{equation}
Replacing $Y=\xi$ in  (\ref{b2}) we deduce
\begin{equation}\label{b4}
\beta\mathcal{L}_{V}(\eta(X))-\beta \eta(\mathcal{L}_{V}X)-\beta g(X, \mathcal{L}_{V}\xi)=(\lambda-r)g(X,\xi),
\end{equation}
or equivalently
\begin{equation}\label{b5}
\beta(\mathcal{L}_{V}\eta)X-\beta g(X, \mathcal{L}_{V}\xi)=(\lambda-r)g(X,\xi).
\end{equation}
In view of (\ref{cvf}), (\ref{b5}) yields
\begin{equation}\label{b6}
\beta\psi \eta(X)-\beta g(X, \mathcal{L}_{V}\xi)=(\lambda-r)g(X,\xi).
\end{equation}
This implies
\begin{equation}\label{b7}
\beta\mathcal{L}_{V}\xi=(\beta \psi-\lambda+r)\xi,
\end{equation}
and
\begin{equation}\label{b8}
\beta g(\mathcal{L}_{V}\xi,\xi)=\beta \psi-\lambda+r.
\end{equation}
Applying (\ref{b3}) in (\ref{b8}) we infer
\begin{equation}\label{b9}
\beta\psi=\frac{3}{2}(\lambda-r).
\end{equation}
Making use of (\ref{b9}) in (\ref{b7}) we obtain
\begin{equation}\label{b10}
\beta \mathcal{L}_{V}\xi=\frac{\lambda-r}{2}\xi.
\end{equation}
From $ \phi(\xi)=0$ and (\ref{b10}) we conclude
\begin{equation}\label{b11}
\beta(\mathcal{L}_{V} \phi)\xi=\beta\mathcal{L}_{V}(\phi\xi)-\phi(\beta\mathcal{L}_{V}\xi)=0.
\end{equation}
On the other hand, on almost co-K\"{a}hler manifold $(M^{2n+1},\phi,\xi,\eta, g)$ we have
\begin{equation}\label{b12}
(\mathcal{L}_{V}d\eta)(X,Y)=(\mathcal{L}_{V} g)(X,\phi Y)+g(X, (\mathcal{L}_{V}\phi)Y),
\end{equation}
for all vector fields $X,Y$ on $M$. Multiplying both sides of (\ref{b12}) in $\beta$ and using (\ref{b1}), we can write
\begin{equation}\label{b13}
\beta(\mathcal{L}_{V}d\eta)(X,Y)=(\lambda-r)g(X,\phi Y)+\beta g(X, (\mathcal{L}_{V}\phi)Y),
\end{equation}
for any vector fields $X,Y$ on $M$. Since $V$ is a contact vector field, from (\ref{cvf}) we find
\begin{equation}\label{b14}
\mathcal{L}_{V}d\eta=d\mathcal{L}_{V}\eta=d(\psi \eta)=(d\psi)\wedge\eta+\psi(d\eta).
\end{equation}
 This gives
\begin{equation}\label{b15}
(\mathcal{L}_{V}d\eta)(X,Y)=\frac{1}{2}\left[d\psi(X)\eta(Y)-d\psi(Y)\eta(X) \right]+\psi g(X, \phi Y).
\end{equation}
In view of (\ref{b13}) and (\ref{b15}) we get
\begin{eqnarray}\nonumber
&&\beta d\psi(X)\eta(Y)-\beta d\psi(Y)\eta(X)+2\beta\psi g(X, \phi Y)\\\label{b16}&&\,\,\,\,\,\,\,\,\,\,\,\,\,\,\,\,\,\,=2(\lambda-r)g(X,\phi Y)+2\beta g(X, (\mathcal{L}_{V}\phi)Y),
\end{eqnarray}
from which it follows that
\begin{equation}\label{b17}
2\beta(\mathcal{L}_{V}\phi)Y=2\big(\beta \psi-\lambda+r\big)\phi Y+\beta\big(\eta(Y)D\psi-(Y\psi)\xi\big).
\end{equation}
Substituting $Y=\xi$ in (\ref{b17}) we obtain
\begin{equation}\label{b18}
2\beta(\mathcal{L}_{V}\phi)\xi=\beta\big(D\psi-(\xi\psi)\xi\big).
\end{equation}
Replacing (\ref{b11}) in (\ref{b18}) we derive
\begin{equation}\label{b19}
D\psi=(\xi\psi)\xi,
\end{equation}
where we use $\beta$ is defined signal. Taking inner product of (\ref{b19}) with respect to any vector field $Y$ we have
$ d\psi(Y)=(\xi \psi)\eta(Y)$,
then
\begin{equation}\label{b20}
d\psi=(\xi \psi)\eta.
\end{equation}
Taking exterior derivative of (\ref{b20}) we conclude
\begin{equation}\label{b21}
0=d^{2}\psi=d(\xi\psi)\wedge \eta+(\xi\psi)d\eta.
\end{equation}
The wedge product of both sides of (\ref{b21}) with $\eta$ implies
\begin{equation}\label{b22}
(\xi\psi)\eta \wedge d\eta=0.
\end{equation}
As $\Omega=\eta\wedge d\eta^{n}$ is the volume form, then $\eta \wedge d\eta\neq0$ and the above equation yields $\xi\psi=0$. Hence from (\ref{b20}) it follows that  $d\psi=0$, thus $\psi$ becomes constant. Tracing (\ref{b1}) over $X,Y$, gives
\begin{equation}\label{b23}
\beta div V=(\lambda-r)(2n+1).
\end{equation}
Taking Lie derivative of the volume form $\Omega=\eta\wedge d\eta^{n}$ along the vector field $V$ and applying the formula $\mathcal{L}_{V}\Omega=(div V)\Omega$ and (\ref{b14}) we obtain
$(div V)\Omega=(n+1)\psi\Omega$ and hence
\begin{equation}\label{b24}
div V=(n+1)\psi.
\end{equation}
From (\ref{b23}) and (\ref{b24}) we have
\begin{equation}\label{b25}
\lambda=r+\frac{n+1}{2n+1}\beta \psi.
\end{equation}
Also, if manifold $M$ be closed then integrating (\ref{b24}) and use divergence theorem we get $\psi=0$ and $\lambda=r$.

\end{proof}
\begin{definition}(\cite{MB}) A vector field $V$ is called torse forming if it satisfies
\begin{equation}\label{tf}
\nabla_{X}V=fX+\theta(X)V
\end{equation}
for all vector field $X$ on $M$, where $f\in C^{\infty}(M)$ and $\theta$ is a $1$-form. A torse forming vector field $V$ is said to be recurrent if $u=0$.
\end{definition}
\begin{definition}(\cite{YC})
A vector field $V$ is called concurrent vector field if $\nabla_{X}V=0$ for any vector field $X$ on $M$.
\end{definition}
\begin{definition}(\cite{BYC})
A nowhere zero vector field $V$ on Riemannian  manifold is called  a torqued vector field if it satisfies
\begin{equation}\label{tv}
\nabla_{X}V=fX+\theta(X)V\,\,\,\,\,\,\,\,\,\text{and}\,\,\,\,\,\,\theta(V)=0.
\end{equation}
\end{definition}
\begin{theorem}\label{t6}
If an  almost co-K\"{a}hler manifold $(M^{2n+1},\phi,\xi,\eta, g)$   admits a $\beta$-almost Yamabe soliton    with $V$ as the potential vector field, $V$ is a  torse forming, then  this soliton is shrinking, steady or expanding according as $\frac{2}{2n+1}\beta \theta(V)+r+2f \beta$ be negative, zero or positive, respectively. Moreover, if $V$ is torqued vector field then $\lambda=r+2f\beta$.
\end{theorem}
\begin{proof}
Let the metric $g$ of  $(\kappa,\mu)$-almost co-K\"{a}hler manifold \linebreak $(M^{2n+1},\phi,\xi,\eta, g)$ be a $\beta$-almost Yamabe soliton    with the potential vector field $V$. Then from (\ref{tv}), for any vector fields $X,Y$ on $M$, we have
\begin{eqnarray}\nonumber
(\lambda-r)g(X,Y)&=&\beta(\mathcal{L}_{V}g)(X,Y)=\beta g(\nabla_{X}V,Y)+ \beta g(X,\nabla_{Y}V)\\\label{10}
&=&2\beta f g(X,Y)+\beta \theta(X)g(V,Y)+\beta \theta(Y)g(V,X).
\end{eqnarray}
Taking contraction of (\ref{10}) over $X$ and $Y$ we get
\begin{equation*}
(\lambda-r-2f\beta)(2n+1)=2\beta\theta(V).
\end{equation*}
If $V$ is  torqued vector field the $\theta(V)=0$. This completes the proof of theorem.
\end{proof}
\section{$\beta$-almost Ricci  solitons on almost co-K\"{a}hler manifolds}
In this section we study the $\beta$-almost Ricci  solitons on almost co-K\"{a}hler manifolds which the potential vector field of soliton is the Reeb vector field.

\begin{theorem}\label{t4}
If a   $(\kappa,\mu)$-almost co-K\"{a}hler manifold $(M^{2n+1},\phi,\xi,\eta, g)$   admits a $\beta$-almost Ricci soliton    with $\xi$ as the potential vector field  and $\beta$ is defined signal, then $\xi$ is a geodesic vector field and  the soliton is  expanding, steady or shrinking according as $\kappa$ is negative, zero  or positive.
\end{theorem}
\begin{proof}
Since $(\kappa,\mu)$-almost co-K\"{a}hler manifold $(M^{2n+1},\phi,\xi,\eta, g)$ is a $\beta$-almost Ricci soliton  with  the potential vector field $\xi$, then for any vector fields $X,Y$ on $M$ we get
\begin{equation*}
\beta(\mathcal{L}_{\xi}g)(X,Y)+2S(X,Y)+2\lambda g(X,Y)=0.
\end{equation*}
The definition of Lie-derivative implies
\begin{equation}\label{06}
\beta g(\nabla_{X}\xi,Y)+\beta g(X,\nabla_{Y}\xi)+2S(X,Y)+2\lambda g(X,Y)=0.
\end{equation}
Putting $Y=\xi$ in the above equation, we obtain
\begin{equation}\label{7}
\beta g(\nabla_{X}\xi,\xi)+\beta g(X,\nabla_{\xi}\xi)+2S(X,\xi)+2\lambda g(X,\xi)=0.
\end{equation}
Since, $g(\nabla_{X}\xi,\xi)=0$ and $S(X,\xi)=2n\kappa g(X,\xi)$, the above equation gives
\begin{equation}\label{8}
\beta\nabla_{\xi}\xi=(4n\kappa+2\lambda)\xi.
\end{equation}
Also, if we set $X=\xi$ in (\ref{7}) then we infer $\lambda=-2n\kappa$ and $\nabla_{\xi}\xi=0$. Therefore, $\xi$ is a geodesic vector field and this completes  the proof of theorem.
\end{proof}

\begin{definition}
A contact manifold  $(M^{2n+1},\phi,\xi,\eta, g)$  is said to be $\eta$-Einstein if its Ricci tensor $S$ satisfies
$$S=ag+b \eta\otimes \eta$$
\end{definition}
where $a$ and $b$ are smooth function on $M$.
\begin{theorem}\label{t5}
If a   $(\kappa,\mu)$-almost co-K\"{a}hler manifold $(M^{2n+1},\phi,\xi,\eta, g)$   admits a $\beta$-almost Ricci soliton    with $\xi$ as the potential vector field, $\xi$ is a torse forming, then  $\xi$ is concurrent  and $M$ is $\eta$-Einstein. Moreover, if $f$ is constant then $\kappa\leq0$  and  $f^{2}=-2\kappa$.
\end{theorem}
\begin{proof}
Let the metric $g$ of $(\kappa,\mu)$-almost co-K\"{a}hler manifold \linebreak$(M^{2n+1},\phi,\xi,\eta, g)$ be $\beta$-almost Ricci soliton    with the potential vector field $\xi$  and $\xi$ is a torse forming. Then $\eta(\xi)=1$  and   torse forming of $\xi$ imply
\begin{equation*}
0=g(\nabla_{X}\xi,\xi)=f\eta(X)+\theta(X),
\end{equation*}
for any vector field $X$ on $M$ and hence we have $\theta=-f\eta$. Consequently (\ref{tf}) reduces to
\begin{equation}\label{tf1}
\nabla_{X}\xi=f\left(X-\eta(X)\xi\right),
\end{equation}
for any vector field $X$ on $M$. Equation  (\ref{tf1}) implies that  $\nabla_{X}\xi$ is collinear to $\phi^{2}X$  for all $X$ and hence $d\eta=0$, that is $\eta$ is closed. Using (\ref{tf1}) in (\ref{06}), we obtain
\begin{equation*}
S(X,Y)=-(f\beta+\lambda)g(X,Y)+f\beta\eta(X)\eta(Y),
\end{equation*}
hence manifold $M$ is $\eta$-Einstein. By definition of Ricci curvature tensor, we have
\begin{equation}\label{rci}
R(X,Y)\xi=\nabla_{X}\nabla_{Y}\xi-\nabla_{Y}\nabla_{X}\xi-\nabla_{[X,Y]}\xi.
\end{equation}
Substituting (\ref{tf}) in (\ref{rci}) implies
\begin{equation}\label{rci1}
R(X,Y)\xi=\big(X(f)Y-Y(f)X \big)- \big( X(f)\eta(Y)-Y(f)\eta(X)\big)\xi+f^{2}\big(\eta(X)Y-\eta(Y)X\big).
\end{equation}
Now, the function $f$ be a constant, then
\begin{equation}\label{rci2}
R(X,Y)\xi=f^{2}\big(\eta(X)Y-\eta(Y)X\big),
\end{equation}
and
\begin{equation}\label{rci3}
S(X,\xi)=-2nf^{2}\eta(X).
\end{equation}
From (\ref{q}), we have
\begin{equation}\label{rci4}
S(X,\xi)=4n\kappa\eta(X).
\end{equation}
Comparing (\ref{rci3}) and   (\ref{rci4}) we infer $f^{2}=-2\kappa$ and $\kappa$ is nonpositive.
\end{proof}
\begin{corollary}\label{c1}
If a   $(\kappa,\mu)$-almost co-K\"{a}hler manifold $(M^{2n+1},\phi,\xi,\eta, g)$   admits a $\beta$-almost Ricci soliton    with $\xi$ as the potential vector field, $\xi$ is a recurrent torse forming, then  $\xi$ is concurrent  and Killing vector field.
\end{corollary}
\begin{proof}
Since $\xi$ is  recurrent vector field, therefore $f=0$. So equation   (\ref{tf1}) yields $\nabla_{X}\xi=0$, for all vector field $X$ on $M$, which means  that   $\xi$ is concurrent  vector field. Also,
\begin{equation*}
(\mathcal{L}_{\xi}g)(X,Y)=g(\nabla_{X}\xi,Y)+ g(X,\nabla_{Y}\xi)=0
\end{equation*}
for all vector fields $X,Y$ on $M$, that  means  $\xi$ is  Killing vector field.
\end{proof}



\begin{thebibliography}{99}
\bibitem{EB} D. E. Blair, Contact manifold in Riemannain geometry, Lecture notes in mathematics, 509, Springer-Verlag, Berlin, 1976.
\bibitem{DEB} D. E. Blair, Riemannian geometry on contact and symplectic manifolds, Progress in mathematics, 203, Birkh\"{a}user Boston, Inc, Boston, MA, 2002.

\bibitem{MB} A. M. Bloga, M. Crasmareanu, Torse forming $\eta$-Ricci solitons in almost para-contact $\eta$-Einstein geometry, Filomat, 31(2) (2017), 499-504.

\bibitem{BC} B. Cappellletti-Montano, A. D. Nicola, I. Yudin, A survey on cosymplectic geometry, Rev. Math. Phys. , 25 (2013), 1343002 (2013).
\bibitem{BYC} B. Y. Chen, Rectifying submanifolds of Riemannian manifolds and torqued vector fields, Kragujevac  J. Math., 41(1) (2017), 93-103.
\bibitem{YC} B. Y. Chen, S. Deshmukh, Ricci solitons and concurrent vector fields, Balkan J. Gem. Appl. 20(1) (2015), 14-25.

\bibitem{DC} D. Chinea, M. deleon, J. C. Marrero, Topology of cosymplectic manifolds, J. Math. Pures Appl., 72 (19993), 567-591.

\bibitem{C} J. T. Cho, M. Kimura, Ricci solitons and real hypersurfaces in  a complex space form, Tohoku Math. J., 19 (2014), 13-21.

\bibitem{PD} P. Dacko,  Z. Olszak, On almost cosmplectic $(\kappa, \mu)$-space, Banach center Publ., 69, Polish Acad. Sci. Inst. Math., Warsaw, 2005, 211-220.
\bibitem{HE} H. Endo, Non-existence of almost cosymplectic manifolds satisfying acertain condition,  Tensor (N. S.) 63 (2002), 272-284.
\bibitem{F} D. Friedan, Nonlinear models in $2+\epsilon$ dimensions, Ann. Phys., 163 (1985), 318-419.
\bibitem{GHA}H. Gahremani-Gol, Some results on $h$-almost Ricci solitons, J. Geom. Phys.,  137 ( 2019), 212-216.
\bibitem{JNG} J.N. Gomes, Q. Wang, C. Xia, On the $h$-almost Ricci soliton,  J. Geom. Phys., 114 (2017), 216-222.
\bibitem{G} A. Ghosh, R. Sharma, J. T. Cho, Contact metric manifolds with $\eta$-parallel torsion tensor, Ann. Glob. Anal. Geom., 34 (2008),287-299.

\bibitem{H} R. S. Hamilton, Three manifolds with positive Ricci curvature, J. Differential Geom., 17 (1982), 255-306.

\bibitem{KH} S. K. Hui, S. S. Shuukla, D. Chakraborty, $\eta$-Ricci solitons on $\eta$-Einstein Kenmotsu manifolds, global Journal of advanced  research on classical and modern geometries, 6 (1) (2017), 1-6.
\bibitem{DKP} D. Kar, P. Majhi, Beta-Almost Ricci solitons on almost coK\"{a}hler manifolds, Korean J. Math., 27(3)  (2019), 691-705.
\bibitem{CM} J. C. Marrero, E. Padron, New examples of compact cosymplectic solvmanifolds, Arch. Math., 34 (1998), 337-345.
\bibitem{ZO} Z. Olszak, On almost cosymplectic manifolds, Kodai math. J., 4 (1981), 239-250.
\bibitem{ZOL} Z. Olszak,  On almost cosymplectic manifolds with K\"{a}hlerian leaves, Tensor (N. S.) 46 (1987), 117-124.
\bibitem{SP} S. Pigola, M. Rigoli, M. Rimoldi, A. Setti, Ricci almost solitons, Ann. Scuola Norm. Sup. Pisa Cl. Sci.,  5 (2011), 757-799.
\bibitem{S} R. Sharma,  Certain results on $K$-contact and $(\kappa, \mu)$-contact manifolds, J. Geom., 89 (2008), 138-149.
\bibitem{JS} Y. J. Suh, U. C. De, Yamabe solitons and Ricci solitons on  almost co-K\"{a}hler manifolds, Cand. Math. Bull., 62(3) (2019), 653-661.
\bibitem{YW} Y. Wang, A generalization of Goldberg conjecture for co-K\"{a}hler manifolds, Mediterr. J. Math. 13 (2016), 2679-2690.
\bibitem{W} E. Woolgar, Some applications of Ricci flow in physics, Can. J. Phys., 86 (2008),  645-651.
\bibitem{KY} K. Yano, Integral formulas in Riemannian geometry, New York, Marcel Dekker, 1970.










\end{thebibliography}
\end{document}